\documentclass[11pt, a4paper]{article}
\usepackage{amsmath}
\usepackage{amsthm}
\usepackage{amssymb}
\usepackage{graphicx}

\theoremstyle{plain}
\newtheorem{theorem}{Theorem}[section]
\newtheorem{conjecture}[theorem]{Conjecture}
\newtheorem{problem}[theorem]{Problem}
\newtheorem{corollary}[theorem]{Corollary}

\newtheorem{lemma}[theorem]{Lemma}

\newtheorem{definition}[theorem]{definition}

\newtheorem{claim}{Claim}
\newtheorem{subclaim}{Subclaim}[claim]

\newcommand{\C}{\mathbb{C}}

\newcommand{\fig}[5]{
\begin{figure}[ht]
\begin{center}
\resizebox{#1}{#2}{\includegraphics{#3}}
\end{center}
\caption{#4}
\label{#5}
\end{figure}
}

\title{%
A generalization of heterochromatic graphs%
}

\author{%
Kazuhiro Suzuki%
\footnote{This work was supported by MEXT. KAKENHI 21740085.}
\footnote{Department of Electronics and Informatics Frontier,
Kanagawa University, Yokohama, Kanagawa, 221-8686 Japan.
kazuhiro@tutetuti.jp.}
}

\date{\empty}

\begin{document}
\maketitle

\begin{abstract}
In 2006, Suzuki, and Akbari \& Alipour
independently presented a necessary and sufficient condition
for edge-colored graphs to have a heterochromatic spanning tree,
where a heterochromatic spanning tree is a spanning tree
whose edges have distinct colors.
In this paper,
we propose {\it $f$-chromatic} graphs
as a generalization of heterochromatic graphs.
An edge-colored graph is {\it $f$-chromatic}
if each color $c$ appears on at most $f(c)$ edges.
We also present a necessary and sufficient condition
for edge-colored graphs to have an $f$-chromatic spanning forest
with exactly $m$ components.
Moreover, using this criterion,
we show that
a $g$-chromatic graph $G$ of order $n$ with $|E(G)|>\binom{n-m}{2}$
has an $f$-chromatic spanning forest
with exactly $m$ ($1 \le m \le n-1$) components
if $g(c) \le \frac{|E(G)|}{n-m}f(c)$ for any color $c$.
\\[6pt]
{\bf Keyword(s):}
$f$-chromatic,
heterochromatic,
rainbow,
multicolored,
totally multicolored,
polychromatic,
colorful,
edge-coloring,
k-bounded coloring,
spanning tree,
spanning forest.
\\[6pt]
{\bf MSC2010:}
05C05\footnote{05C05 Trees.},
05C15\footnote{05C15 Coloring of graphs and hypergraphs.}.
\end{abstract}

\section{Introduction}

We consider finite undirected graphs without loops or multiple edges.
Let $G$ be a graph with vertex set $V(G)$ and edge set $E(G)$.
An \textit{edge-coloring} of a graph $G$
is a mapping $color:E(G) \rightarrow \C$,
where $\C$ is a set of colors.
An \textit{edge-colored graph} $(G, \C, color)$
is a graph $G$ with an edge-coloring $color$ on a color set $\C$.
We often abbreviate an edge-colored graph $(G, \C, color)$ as $G$.

An edge-colored graph $G$ is said to be \textit{heterochromatic}
if no two edges of $G$ have the same color,
that is,
$color(e) \ne color(f)$ for any two distinct edges $e$ and $f$ of $G$.
A heterochromatic graph is also said to be 
\textit{rainbow}, \textit{multicolored}, \textit{totally multicolored},
\textit{polychromatic}, or \textit{colorful}.
Heterochromatic subgraphs of edge-colored graphs
have been studied in many papers.
(See the survey by Kano and Li \cite{KaLi08}.)

We begin with some results
for the existence of heterochromatic spanning trees and forests.
Brualdi and Hollingsworth \cite{BrHo96} showed the following theorem and conjecture
for edge-disjoint heterochromatic spanning trees
in complete graphs.

\begin{theorem}[Brualdi and Hollingsworth, (1996) \cite{BrHo96}]
If the complete graph $K_{2n}$ $(n \ge 3)$
is edge-colored in such a way that
each color induces a perfect matching,
then it has two edge-disjoint heterochromatic spanning trees.
\label{thm:BrHo96}
\end{theorem}

\begin{conjecture}[\cite{BrHo96}]
Under the same condition as in Theorem \ref{thm:BrHo96},
the edges of $K_{2n}$ can be partitioned into
$n$ edge-disjoint heterochromatic spanning trees.
\label{conj:BrHo96}
\end{conjecture}

Suzuki \cite{Su06} presented a necessary and sufficient condition
for general connected graphs to have a heterochromatic spanning tree.
Here, we denote by $\omega(G)$ the number of components of a graph $G$.
Given an edge-colored graph $G$ and a color set $R$,
we define $E_{R}(G) = \{ e \in E(G) ~|~ color(e) \in R \}$.
Similarly, for a color $c$,
we define $E_c(G) = E_{\{c\}}(G)$.
We denote the graph $(V(G), E(G) \setminus E_R(G))$ by $G-E_R(G)$.

\begin{theorem}[Suzuki, (2006) \cite{Su06}]
An edge-colored  connected graph $G$
has a heterochromatic spanning tree
if and only if
\begin{equation*}
\omega(G-E_R(G)) \le |R|+1 \text{~~~~~ for any } R \subseteq \C.
\end{equation*}
\label{thm-Su06-1}
\end{theorem}

Jin and Li \cite{JiLi06} generalized this theorem
to the following theorem,
from which we can obtain Theorem \ref{thm-Su06-1} by taking $k=n-1$.

\begin{theorem}[Jin and Li, (2006) \cite{JiLi06}]
An edge-colored  connected graph $G$ of order $n$
has a spanning tree with at least $k$ $(1 \le k \le n-1)$ colors
if and only if
\begin{equation*}
\omega(G-E_R(G)) \le n-k+|R| \text{~~~~~ for any } R \subseteq \C.
\end{equation*}
\label{thm-JiLi06-1}
\end{theorem}

If an edge-colored connected graph $G$ of order $n$
has a spanning tree with at least $k$ colors,
then
$G$ has a heterochromatic spanning forest with $k$ edges,
that is,
$G$ has a heterochromatic spanning forest with exactly $n-k$ components.
On the other hand,
If an edge-colored connected graph $G$ of order $n$
has a heterochromatic spanning forest with exactly $n-k$ components,
then
the forest can be turned into a spanning tree with at least $k$ colors
by adding some $n-k-1$ edges.
Hence, we can rephrase Theorem \ref{thm-JiLi06-1} as the following.

\begin{theorem}[\cite{JiLi06}]
An edge-colored  connected graph $G$ of order $n$
has a heterochromatic spanning forest
with exactly $n-k$ components $(1 \le k \le n-1)$
if and only if
\begin{equation*}
\omega(G-E_R(G)) \le n-k+|R| \text{~~~~~ for any } R \subseteq \C.
\end{equation*}
\label{thm-JiLi06-2}
\end{theorem}

Akbari and Alipour \cite{AkAl06}
gave another necessary and sufficient condition
for graphs to have a heterochromatic spanning tree.

\begin{theorem}[Akbari and Alipour, (2006) \cite{AkAl06}]
An edge-colored connected graph $G$ of order $n$
has a heterochromatic spanning tree
if and only if
for every partition of $V(G)$ into $t$ $(1 \le t \le n)$ parts,
there exist at least $t-1$ edges
with distinct colors that join different partition sets.
\label{thm-AkAl06-1}
\end{theorem}

Theorem \ref{thm-Su06-1} and Theorem \ref{thm-AkAl06-1}
are essentially the same,
but the proofs are different.
Theorem \ref{thm-Su06-1} was proved graph theoretically,
and Theorem \ref{thm-AkAl06-1} was proved
by using Rado's Theorem in Matroid Theory.

Suzuki \cite{Su06} proved the following theorem
by applying Theorem \ref{thm-Su06-1}.

\begin{theorem}[Suzuki, (2006) \cite{Su06}]
An edge-colored  complete graph $K_n$ has a heterochromatic spanning tree
if $|E_c(G)| \le n/2$ for any color $c \in \C$.
\label{thm-Su06-2}
\end{theorem}

This theorem implies that
by properly bounding numbers of edges for each color,
the graph can contain enough colors
to have a heterochromatic spanning tree.
If the edges of a graph $G$
is colored so that no color appears on more than $k$ edges,
we refer to this as a {\it $k$-bounded edge-coloring}.
Erd\H{o}s, Nesetril and R\"{o}dl \cite{ErNeRo83} mentioned
the following problem.

\begin{problem}[Erd\H{o}s, Nesetril and R\"{o}dl, (1983) \cite{ErNeRo83}]
Find a bound $k=k(n)$
such that every $k$-bounded edge-colored complete graph $K_n$
contains a heterochromatic Hamiltonian cycle.
\label{pro-100412-1}
\end{problem}

Hahn and Thomassen \cite{HaTh86} proved the following theorem.

\begin{theorem}[Hahn and Thomassen, (1986) \cite{HaTh86}]
There exists a constant number $c$
such that
if $n \ge ck^3$
then every $k$-bounded edge-colored complete graph $K_n$
has a heterochromatic Hamiltonian cycle.
\label{thm-HaTh86-1}
\end{theorem}

Albert, Frieze, and Reed \cite{AlFrRe95}
improved Theorem \ref{thm-HaTh86-1}.

\begin{theorem}[Albert, Frieze, and Reed, (1995) \cite{AlFrRe95}]
Let $c < 1/32$.
If $n$ is sufficiently large and $k \le \lceil cn \rceil$,
then every $k$-bounded edge-colored complete graph $K_n$
has a heterochromatic Hamiltonian cycle.
\label{thm-AlFrRe95-1}
\end{theorem}

In this paper,
we will propose a generalization of heterochromatic graphs,
which is also a generalization of $k$-bounded colored graphs.
Moreover,
we will generalize
Theorem \ref{thm-Su06-1}, \ref{thm-JiLi06-2}, \ref{thm-Su06-2},
and Problem \ref{pro-100412-1}.

\section{Heterochromatic and $f$-chromatic graphs}

{\it Heterochromatic} or {\it $k$-bounded colored}
means that
any color appears at most {\it once} or $k$ times, respectively.
We propose to generalize {\it once} and $k$ to a mapping $f$
from a given color set $\C$ to the set of non-negative integers.
We introduce the following definition as a generalization
of heterochromatic or $k$-bounded colored graphs%
\footnote{We name it after {\it heterochromatic}.
Of course, we may name it as {\it $f$-bounded colored}.}.

\begin{definition}
Let $f$ be a mapping from a given color set $\C$
to the set of non-negative integers.
An edge-colored graph $(G, \C, color)$ is said to be
\textit{$f$-chromatic}
if $|E_c(G)| \le f(c)$ for any color $c \in \C$. 
\label{def-100219-1}
\end{definition}

Fig.\ref{fig-100629-1} shows an example
of an $f$-chromatic spanning tree of an edge-colored graph.
Let $\C = \{1,2,3,4,5,6,7 \}$ be a given color set of $7$ colors,
and a mapping $f$ is given as follows:
$f(1)=3$,
$f(2)=1$,
$f(3)=3$,
$f(4)=0$,
$f(5)=0$,
$f(6)=1$,
$f(7)=2$.
Then, the left edge-colored graph in Fig.\ref{fig-100629-1}
has the right graph as a subgraph.
It is a spanning tree
where each color $c$ appears at most $f(c)$ times.
Thus, it is an $f$-chromatic spanning tree.

\fig{\textwidth}{!}{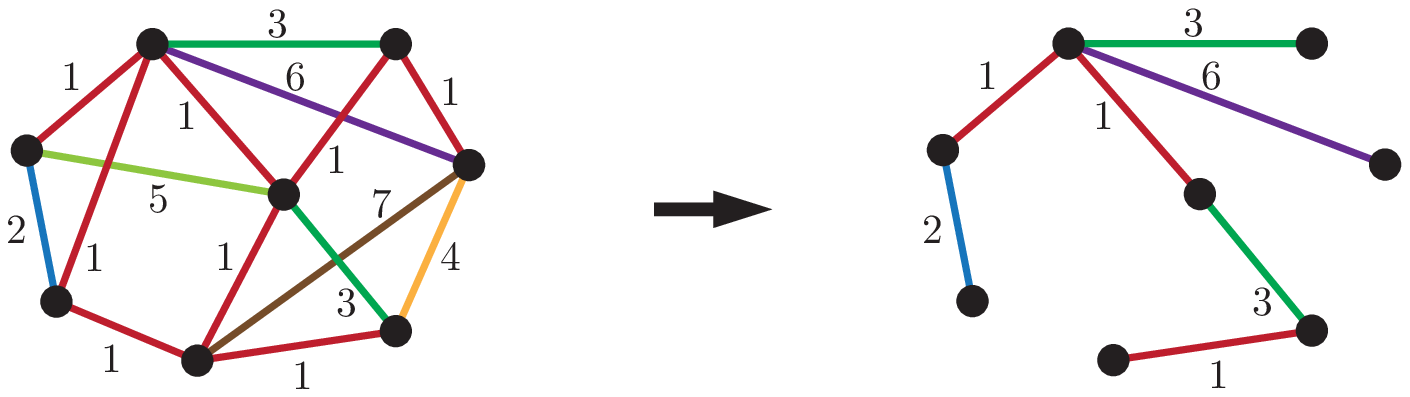}
{An $f$-chromatic spanning tree of an edge-colored graph.}
{fig-100629-1}

If $f(c) = 1$ for any color $c$,
then all $f$-chromatic graphs are heterochromatic
and also all heterochromatic graphs are $f$-chromatic.
We expect many previous studies and results
for heterochromatic subgraphs
will be generalized.
For example,
we give the following generalization of Problem \ref{pro-100412-1}.

\begin{problem}
Find a relationship between two functions $f$ and $g$
such that every $g$-chromatic complete graph $K_n$
contains an $f$-chromatic Hamiltonian cycle
(Hamiltonian path, spanning tree, or other subgraph).
\label{pro-100412-2}
\end{problem}

In this paper,
we generalize Theorems \ref{thm-Su06-1}, \ref{thm-JiLi06-2}, and \ref{thm-Su06-2}.
Let $\C$ be a color set,
and $f$ be a mapping from $\C$ to the set of non-negative integers.
We present the following necessary and sufficient condition
for graphs to have an $f$-chromatic spanning forest
with exactly $m$ components.

\begin{theorem}
An edge-colored graph $(G, \C, color)$ of order at least $m$
has an $f$-chromatic spanning forest with exactly $m$ components
if and only if
\begin{equation*}
\omega(G-E_R(G)) \le m+\sum_{c \in R}f(c) \text{~~~~~ for any } R \subseteq \C.
\end{equation*}
\label{thm-091015-1}
\end{theorem}

Note that, it is allowed $R = \emptyset$,
that is,
the condition includes the necessary condition
that $\omega(G) \le m$ to have a spanning forest with $m$ components.
From this theorem,
we can obtain the following corollary.

\begin{corollary}
Let $\C=\{ c_1,c_2,\ldots,c_r \}$
and $f$ be a mapping from $\C$ to the set of non-negative integers
such that $\sum_{i=1}^{r}f(c_i)=n-m$.
An edge-colored graph $(G, \C, color)$ of order $n$
has an spanning  forest with exactly $m$ components
and exactly $f(c_i)$ edges for each color $c_i$
if and only if
\begin{equation*}
\omega(G-E_R(G)) \le m+\sum_{c \in R}f(c) \text{~~~~~ for any } R \subseteq \C.
\end{equation*}
\label{cor-100713-1}
\end{corollary}

This corollary is interesting
because in the corollary
we can not only desire an $f$-chromatic spanning forest,
but also fix the exactly number of edges for each color.

Theorem \ref{thm-091015-1} will be proved in the next section.
By applying Theorem \ref{thm-091015-1},
we will prove the following theorem.

\begin{theorem}
A $g$-chromatic graph $G$ of order $n$ with $|E(G)|>\binom{n-m}{2}$
has an $f$-chromatic spanning forest with exactly $m$ $(1 \le m \le n-1)$ components
if $g(c) \le \frac{|E(G)|}{n-m}f(c)$ for any color $c$.
\label{thm-100503-1}
\end{theorem}

In order to prove Theorem \ref{thm-100503-1},
we need the following Lemma.

\begin{lemma}
Let $G$ be a graph of order $n$
that consists of $s$ components.
Then $|E(G)| \le \binom{n-(s-1)}{2}$.
\label{lem-110208-1}
\end{lemma}

\begin{proof}
Take a graph $G^*$ with the maximum number of edges
that satisfies the condition in the Lemma.
By the maximality of $G^*$,
each component is complete.
Let $D_s$ be a maximum component of $G^*$.
Suppose that
some component $D$ except $D_s$ has at least two vertices.
Let $x$ be a vertex of $D$.
Let $D' = D-x$
and $D'_s = (V(D_s) \cup \{x \}, E(D_s) \cup \{xz ~|~ z \in D_s \})$.
Then, we have
$|E(D)|+|E(D_s)| < |E(D')|+|E(D'_s)|$,
which contradicts the maximality of $E(G^*)$.
Thus, every component except $D_s$ has exactly one vertex,
which implies that
$|V(D_s)|=n-(s-1)$.
Therefore,
$|E(G)| \le |E(G^*)| = \binom{n-(s-1)}{2}$.
\end{proof}

We here prove Theorem \ref{thm-100503-1}
using Lemma \ref{lem-110208-1}.

\begin{proof}
Suppose that $G$ has no $f$-chromatic spanning forests
with exactly $m$ components.
By Theorem \ref{thm-091015-1},
there exists a color set $R \subseteq \C$
such that
\begin{equation*}
\omega(G-E_R(G)) > m+\sum_{c \in R}f(c).
\end{equation*}

Let $s = \omega(G-E_R(G))$ and $r = \sum_{c \in R}f(c)$.
Then we have
\begin{equation}
n \ge s > m+r.
\label{eq-100502-1}
\end{equation}

Let $D_1, D_2, \ldots, D_s$
be the components of $G-E_R(G)$,
and $q$ be the number of edges of $G$ between these distinct components.
Note that,
the colors of these $q$ edges are only in $R$.
By the assumption on the function $g$,
we get 
\begin{equation}
q \le \sum_{c \in R}g(c)
\le \sum_{c \in R}\frac{|E(G)|}{n-m}f(c)
= \frac{|E(G)|}{n-m}r.
\label{eq-100502-2}
\end{equation}

On the other hand,
\begin{eqnarray*}
q &=& |E(G)| - |E(D_1) \cup E(D_2) \cup \cdots \cup E(D_s)|.
\end{eqnarray*}
By Lemma \ref{lem-110208-1},
$|E(D_1) \cup E(D_2) \cup \cdots \cup E(D_s)| \le \binom{n-(s-1)}{2}$.
Thus, since $s \ge r+m+1$ by (\ref{eq-100502-1}), we have
\begin{eqnarray*}
q &=& |E(G)| - |E(D_1) \cup E(D_2) \cup \cdots \cup E(D_s)|\\
  &\ge & |E(G)| - \binom{n-(s-1)}{2}
    \ge |E(G)| - \binom{n-(r+m)}{2}\\
  &=& |E(G)| - \frac{(n-r-m)(n-r-m-1)}{2}.
\end{eqnarray*}
Hence, by (\ref{eq-100502-2}),
\begin{equation*}
|E(G)| - \frac{(n-r-m)(n-r-m-1)}{2} \le \frac{|E(G)|}{n-m}r,
\end{equation*}
and thus,
\begin{equation*}
\frac{n-m-r}{n-m}|E(G)| \le \frac{(n-r-m)(n-r-m-1)}{2}.\\
\end{equation*}
Since $n-r-m>0$ by (\ref{eq-100502-1}),
by dividing both sides of the equation by $n-r-m$
\begin{equation*}
\frac{|E(G)|}{n-m} \le \frac{n-r-m-1}{2},
\end{equation*}
and thus,
\begin{equation*}
|E(G)| \le \frac{(n-m)(n-r-m-1)}{2}.\\
\end{equation*}
Since $r \ge 0$, we get $|E(G)| \le \binom{n-m}{2}$,
which contradicts our assumption.
Therefore, $G$ has an $f$-chromatic spanning forest with $m$ components.
\end{proof}

We can obtain the following corollary from Theorem \ref{thm-100503-1}.

\begin{corollary}
A $g$-chromatic complete graph $K_n$
has an $f$-chromatic spanning forest
with exactly $m$ $(1 \le m \le n-1)$ components
if $g(c) \le \frac{n(n-1)}{2(n-m)}f(c)$ for any color $c$.
\label{cor-091015-2}
\end{corollary}

Theorems \ref{thm-Su06-1} and \ref{thm-Su06-2} are special cases of
Theorem \ref{thm-091015-1} and Corollary \ref{cor-091015-2}
with $m=1$ and $f(c)=1$.
Theorem \ref{thm-JiLi06-2} is a special case
of Theorem \ref{thm-091015-1} for connected graphs
with $m=n-k$ and $f(c)=1$.
The proof of Theorem \ref{thm-100503-1} is essentially the same as
the proof of Theorem \ref{thm-Su06-2} in \cite{Su06}.
We can also prove Theorem \ref{thm-091015-1}
in a similar way as the proof of Theorem \ref{thm-Su06-1} in \cite{Su06}.
However, in this paper,
we will improve the proof
by introducing the new notion of \textit{Saturated Conditions}.

\section{Proof of Theorem \ref{thm-091015-1}}

We begin with some notation.
We use the symbol $\subset$ to denote proper inclusion.
We often denote an edge $e = \{ x,y \}$ by $xy$ or $yx$.
For a graph $G$ and a subset $E \subseteq E(G)$,
we denote the graphs
$(V(G), E(G) \setminus E)$ and $(V(G), E(G) \cup E)$
by $G-E$ and $G+E$, respectively.
Similarly,
for an edge $e$ whose end vertices are in $V(G)$,
we denote the graphs
$(V(G), E(G) \setminus \{ e \})$ and $(V(G), E(G) \cup \{ e \})$
by $G-e$ and $G+e$, respectively.
For an edge-colored graph $(G, \C, color)$ and an edge set $E \subseteq E(G)$,
we define $color(E) = \{ color(e) ~|~ e \in E \}$
and $color(G) = color(E(G))$.

First, we prove the necessity.
Let $F$ be an $f$-chromatic spanning forest of $G$
with exactly $m$ components.
Consider $G-E_R(G)$ for any color subset $R \subseteq \C$.
Since $F - E_R(F)$ is a spanning forest of $G-E_R(G)$,
we have $\omega(G-E_R(G)) \le \omega(F - E_R(F))$.
Moreover, $\omega(F - E_R(F)) = m+|E_R(F)|$
because $F$ is a forest with exactly $m$ components.
By the definition of $f$-chromatic graphs,
$|E_c(F)| \le f(c)$ for any color $c$.
Thus, we have
\begin{equation*}
|E_R(F)| = \sum\limits_{c \in R}|E_c(F)| \le \sum\limits_{c \in R}f(c).
\end{equation*}
Hence,
\begin{equation*}
\omega(G-E_R(G)) \le \omega(F - E_R(F)) = m+|E_R(F)| \le m+\sum\limits_{c \in R}f(c).
\end{equation*}

Next, we prove the sufficiency by contradiction.
Suppose that $G$ has no $f$-chromatic spanning forests with exactly $m$ components.

\begin{claim}
Any $f$-chromatic spanning forest of $G$ has at least $m+1$ components.
\label{claim-091216-1}
\end{claim}
\begin{proof}
If there exists an $f$-chromatic spanning forest of $G$
with at most $m-1$ components,
then it can be turned into an $f$-chromatic spanning forest of $G$
with exactly $m$ components
by removing edges one by one,
which contradicts our assumption.
\end{proof}

We denote by $E^*_F$ the set of edges
between the components of a forest $F$ in $G$,
namely,
\begin{equation*}
E^*_F := \{ xy \in E(G) | \text{ $x$ and $y$ are not in the same components of $F$} \}.
\end{equation*}

Fig.\ref{fig-100702-1} shows
an example of $E^*_F$.
Let $D_1, D_2, \ldots, D_7$ be the components of $G - E^*_F$,
which is induced by the components of $F$.
We simplify the left graph to the right illustration.

\fig{\textwidth}{!}{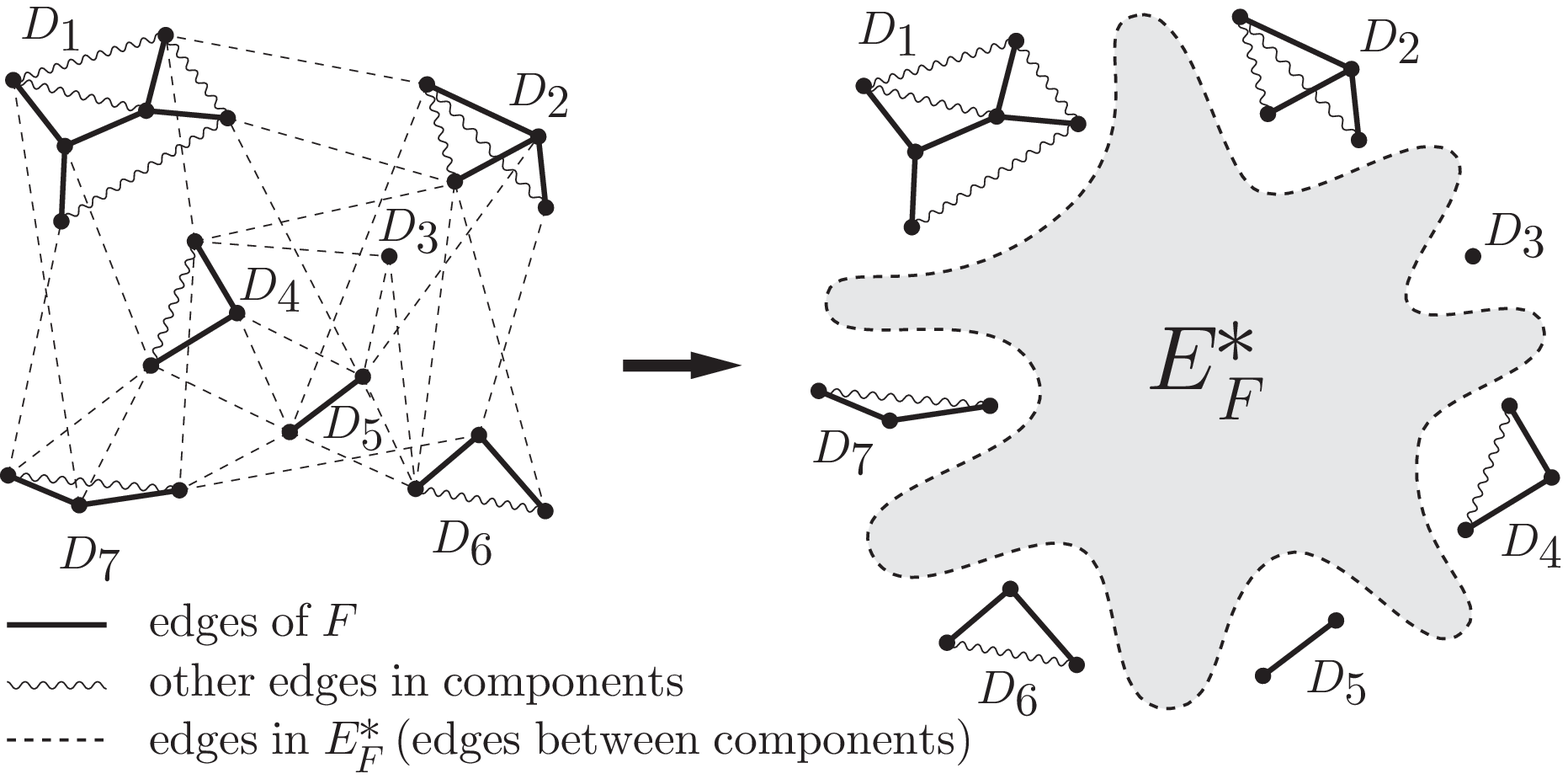}
{An example of $E^*_F$, which is the set of edges
between the components of a forest $F$ in $G$.}
{fig-100702-1}

For an $f$-chromatic spanning forest $F$ of $G$ and two sets of colors $C_0, C_1$,
the triple $<F,C_0,C_1>$ is said to be \textit{saturated}
if the following conditions hold:
\begin{enumerate}
\item[] \underline{Saturated Conditions}
\begin{enumerate}
\item[(i)] $C_0 \cap C_1 = \emptyset$,
\item[(ii)] $C_0 \cup C_1 = color(E^*_F)$,
\item[(iii)] $C_0 \cap color(F) = \emptyset$,
\item[(iv)] $\omega(F) \ge m+1+\sum\limits_{c \in C_0}f(c)$,
\item[(v)] $|E_c(\tilde{F})| = f(c)$ for every color $c \in C_1$,
where $\tilde{F}$ is any $f$-chromatic spanning forest  of $G$
such that $E^*_{\tilde{F}} = E^*_F$ and $C_0 \cap color(\tilde{F}) = \emptyset$.
\end{enumerate}
\end{enumerate}

Saturated Conditions imply that
if already $\tilde{F}$ has $f(c)$ edges for every color $c \in C_1$,
then we can not add more edges in $E^*_{\tilde{F}}$
whose color appears on $\tilde{F}$
in order to get a larger $f$-chromatic spanning forest,
namely, we call that \textit{saturated}.
Note that,
it is allowed that
$\tilde{F} = F$ in the condition (v).

\begin{claim}
There exists a saturated triple $<F, C_0, C_1>$ in $G$.
\label{claim-091216-2}
\end{claim}
\begin{proof}
$G$ has an $f$-chromatic spanning forest,
because the graph $(V(G), \emptyset)$ is an $f$-chromatic spanning forest.
Let $F$ be an $f$-chromatic spanning forest with minimum $\omega(F)$,
and let
\begin{eqnarray*}
C_0 &:=& color(E^*_F) \setminus color(F),\\
C_1 &:=& color(E^*_F) \setminus C_0.
\end{eqnarray*}
Then, the triple $<F, C_0, C_1>$ satisfies
the saturated conditions (i), (ii) and (iii).
Suppose that there exists a color $c \in C_0$ such that $f(c) \ge 1$.
By the definition of $C_0$,
there exists some edge $e \in E^*_F$ with the color $c$.
By the saturated condition (iii),
the spanning forest $F+e$ is $f$-chromatic,
which contradicts the minimality of $\omega(F)$.
Thus, $f(c) = 0$ for any color $c \in C_0$.
Then, by Claim \ref{claim-091216-1}, we have
\begin{eqnarray*}
\omega(F) \ge m+1 = m+1+0 = m+1+\sum\limits_{c \in C_0}f(c).
\end{eqnarray*}
Hence, the triple $<F, C_0, C_1>$ satisfies
the saturated condition (iv).

Let $\tilde{F}$ be any $f$-chromatic spanning forest  of $G$
such that $E^*_{\tilde{F}} = E^*_F$ and $C_0 \cap color(\tilde{F}) = \emptyset$.
Note that $\tilde{F}$ may be $F$.
By the definition of $f$-chromatic graphs,
$|E_c(\tilde{F})| \le f(c)$ for any color $c \in C_1 \subseteq \C$.
Suppose that there exists a color $c \in C_1$ such that $|E_c(\tilde{F})| < f(c)$.
By the definition of $C_1$,
there exists some edge $e \in E^*_F$ with the color $c$.
Since $E^*_{\tilde{F}} = E^*_F$,
the edge $e$ is also in $E^*_{\tilde{F}}$.
Then, the spanning forest $\tilde{F}+e$ is $f$-chromatic,
which contradicts the minimality of $\omega(F)$.
Thus, $|E_c(\tilde{F})| = f(c)$ for any color $c \in C_1$.
Hence, the triple $<F, C_0, C_1>$ satisfies
the saturated condition (v).
Therefore, $<F, C_0, C_1>$ is a saturated triple in $G$.
\end{proof}

Let $<F, C_0, C_1>$ be a saturated triple with maximal $C_0$ in $G$.

\begin{claim}
$C_1 \ne \emptyset$.
\label{claim-110221-1}
\end{claim}
\begin{proof}
Suppose $C_1 = \emptyset$.
By the saturated condition (ii) of $<F, C_0, C_1>$,
$C_0 = C_0 \cup C_1 = color(E^*_F)$.
Then, $E^*_F \subseteq E_{C_0}(G)$.
Hence, by the saturated condition (iv) of $<F, C_0, C_1>$, we have
\begin{eqnarray*}
\omega(G-E_{C_0}(G)) \ge \omega(G-E^*_F) = \omega(F) \ge m+1+\sum\limits_{c \in C_0}f(c),
\end{eqnarray*}
which contradicts the assumption of the theorem.
\end{proof}

By the definition of a saturated triple,
$F$ is an $f$-chromatic spanning forest of $G$.
We define a triple $<F', C'_0, C'_1>$ based on $<F, C_0, C_1>$ as follows:
\begin{eqnarray*}
F' &:=& F-E_{C_1}(F),\\
C'_0 &:=& C_0 \cup C_1,\\
C'_1 &:=& color(E^*_{F'})\setminus C'_0.
\end{eqnarray*}

Fig.\ref{fig-100702-2} shows
an example of $E^*_F$ and $E^*_{F'}$
where the left illustration is the same as in Fig.\ref{fig-100702-1}.
By removing edges in $E_{C_1}(F)$ from $F$,
some components of $F$ splits into several new components.
$E^*_{F'}$ is the set of edges between these new components
and edges in $E^*_F$,
that is,
the set of edges between components of $F'$ in $G$.

\fig{\textwidth}{!}{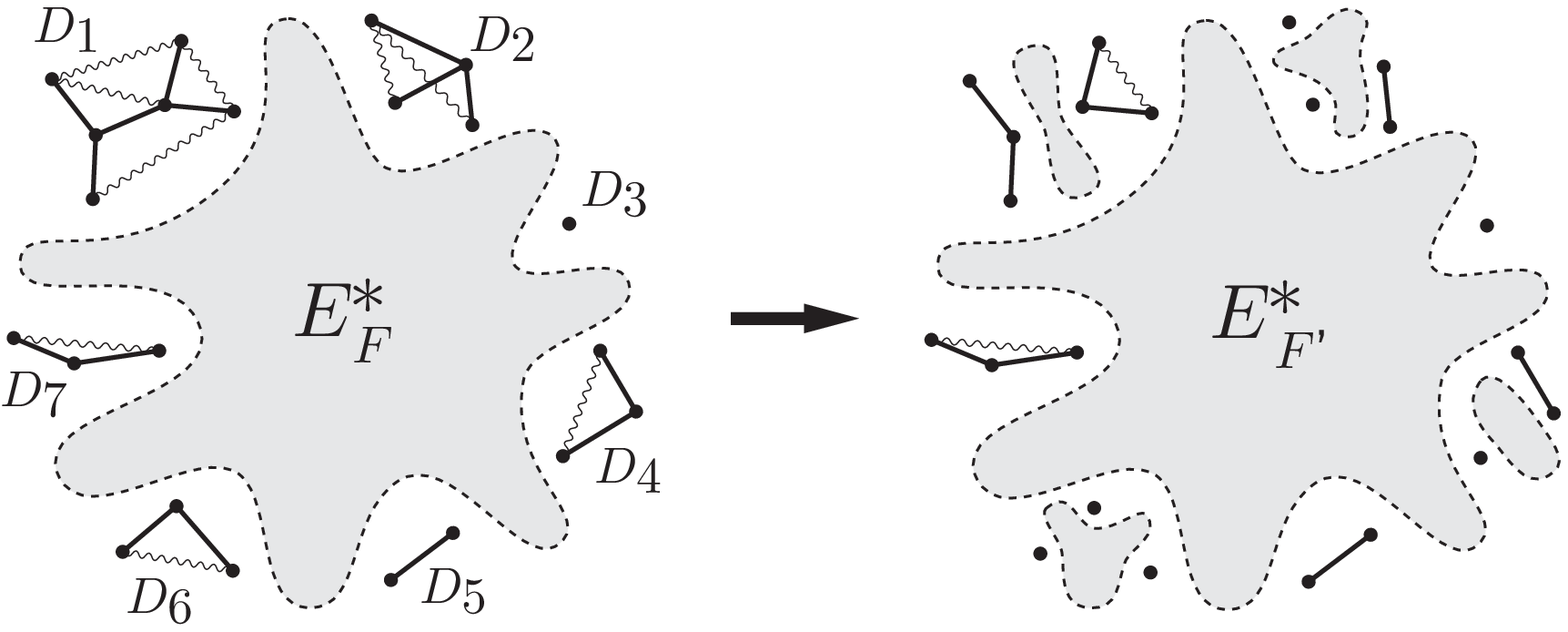}
{An example of $E^*_F$ and $E^*_{F'}$.}
{fig-100702-2}

Since $C_1 \ne \emptyset$ by Claim \ref{claim-110221-1},
we have $C_0 \subset C_0 \cup C_1 = C'_0$,
that is, $C'_0$ properly contains $C_0$.
Thus,
if we can prove that $<F', C'_0, C'_1>$ is saturated
then it contradicts the maximality of $C_0$
and Theorem \ref{thm-091015-1} is proved.
Note that $F'$ is an $f$-chromatic spanning forest of $G$
because $F$ is an $f$-chromatic spanning forest of $G$.

\begin{claim}
$<F', C'_0, C'_1>$ satisfies
the saturated conditions (i), (ii), and (iii).
\label{claim-100704-1}
\end{claim}
\begin{proof}
The triple $<F', C'_0, C'_1>$ satisfies
the saturated condition (i),
that is, $C'_0 \cap C'_1 = \emptyset$
by the definition of $C'_0$ and $C'_1$.

By the definition of $F'$,
we have $color(E^*_F) \subseteq color(E^*_{F'})$.
By the definition of $C'_0$
and the condition (ii) of the saturated triple $<F, C_0, C_1>$,
we have $C'_0 = C_0 \cup C_1 = color(E^*_F) \subseteq color(E^*_{F'})$.
Thus, by the definition of $C'_1$,
the triple $<F', C'_0, C'_1>$ satisfies
the saturated condition (ii).

By the condition (iii) of the saturated triple $<F, C_0, C_1>$,
we have $C_0 \cap color(F) = \emptyset$,
so $C_0 \cap color(F') = \emptyset$
because $color(F') \subseteq color(F)$.
By the definition of $F'$,
we have $C_1 \cap color(F') = \emptyset$.
Hence, by the definition of $C'_0$,
the triple $<F', C'_0, C'_1>$ satisfies
the saturated condition (iii).
\end{proof}

\begin{claim}
$<F', C'_0, C'_1>$ satisfies
the saturated condition (iv).
\label{claim-100704-2}
\end{claim}
\begin{proof}
By the definition of $F'$,
we have
\begin{eqnarray}
\omega(F')
&=
& \omega(F) + \sum\limits_{c \in C_1}|E_c(F)|.
\label{inequality-110222-1}
\end{eqnarray}
By the saturated conditions (iv) and (v) of $<F, C_0, C_1>$,
we have
\begin{eqnarray}
\omega(F) + \sum\limits_{c \in C_1}|E_c(F)|
& \ge
& m+1+\sum\limits_{c \in C_0}f(c) + \sum\limits_{c \in C_1}f(c).
\label{inequality-110222-2}
\end{eqnarray}
By the saturated condition (i) of $<F, C_0, C_1>$
and the definition of $C'_0$,
we have
\begin{eqnarray}
\sum\limits_{c \in C_0}f(c) + \sum\limits_{c \in C_1}f(c)
& =
& \sum\limits_{c \in C_0 \cup C_1}f(c)
= \sum\limits_{c \in C'_0}f(c).
\label{inequality-110222-3}
\end{eqnarray}
Thus, by the equations and inequalities
(\ref{inequality-110222-1}),
(\ref{inequality-110222-2}), and
(\ref{inequality-110222-3}),
we have
\begin{eqnarray*}
\omega(F')
& \ge
& m+1+\sum\limits_{c \in C'_0}f(c).
\end{eqnarray*}
Hence, the triple $<F', C'_0, C'_1>$ satisfies
the saturated condition (iv).
\end{proof}

In order to prove the last saturated condition (v),
we need some preparation.
Let $\tilde{F'}$ be any $f$-chromatic spanning forest of $G$
such that $E^*_{\tilde{F'}} = E^*_{F'}$
and $C'_0 \cap color(\tilde{F'}) = \emptyset$.
By the definition of $F'$,
$F = F' + E_{C_1}(F)$.
We want to consider the graph $H = \tilde{F'} + E_{C_1}(F)$ instead of $F$.

Fig.\ref{fig-100704-1} shows
how to construct the graph $H$ from $F$.
First, we get $F'$ by removing the edges in $E_{C_1}(F)$ from $F$.
Next, by changing edges only inside components of $G-E^*_{F'}$,
we pick up any $f$-chromatic spanning forest $\tilde{F'}$ of $G$
such that $E^*_{\tilde{F'}} = E^*_{F'}$
and $C'_0 \cap color(\tilde{F'}) = \emptyset$.
Last, we get $H$ by adding back the edges in $E_{C_1}(F)$ to $\tilde{F'}$,
which are indicated by double lines.

\fig{\textwidth}{!}{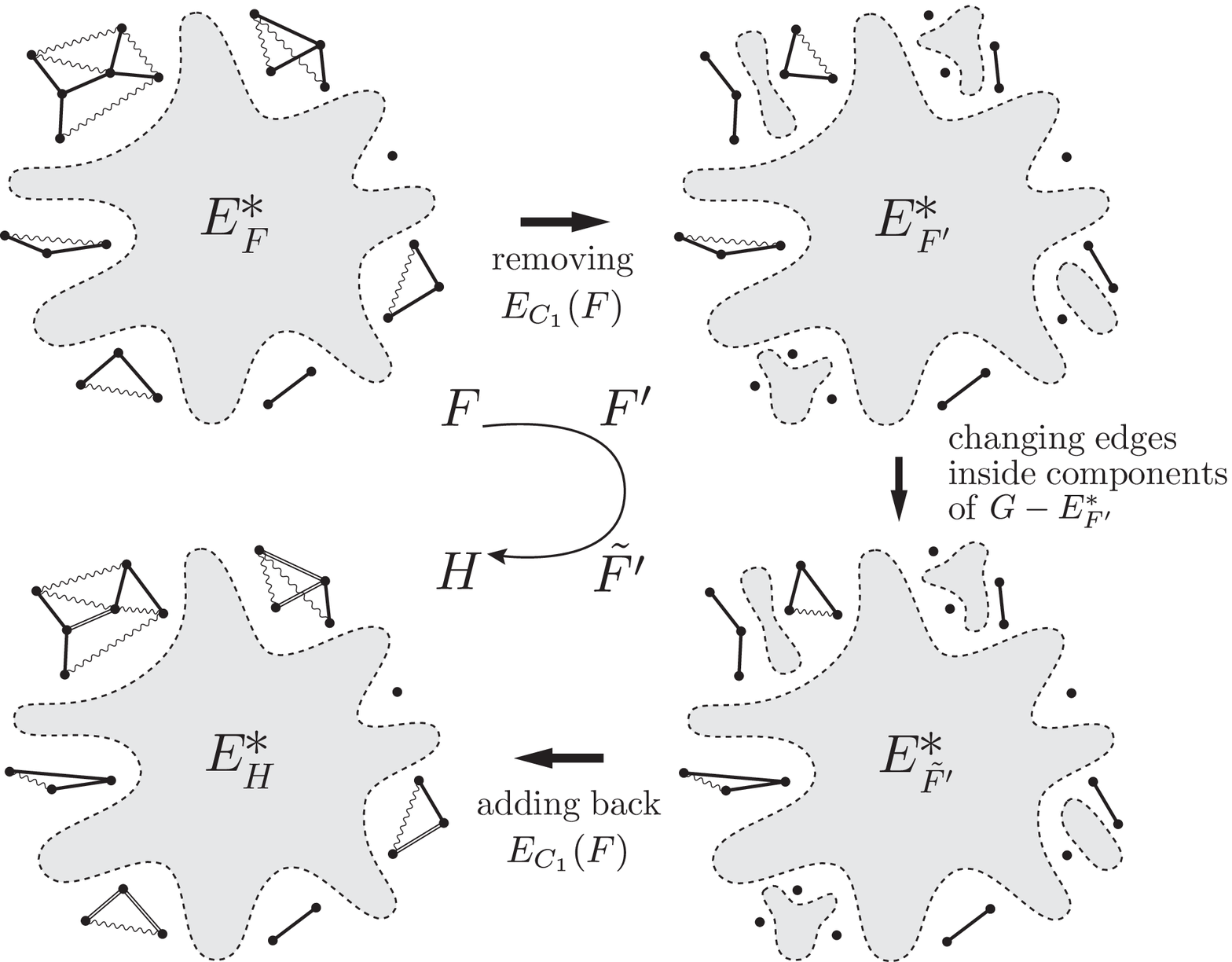}
{How to construct $H$ from $F$.}
{fig-100704-1}

\begin{claim}
$H$ is an $f$-chromatic spanning forest of $G$
such that $E^*_H = E^*_F$ and $C_0 \cap color(H) = \emptyset$.
\label{claim-100704-3}
\end{claim}
\begin{proof}
$H$ is a spanning subgraph of $G$
because $\tilde{F'}$ is a spanning forest of $G$.
Let $T_1, T_2, \ldots, T_k$
and $\tilde{T}_1, \tilde{T}_2, \ldots, \tilde{T}_l$
be the components of $F'$ and $\tilde{F'}$, respectively, which are trees.
The graphs $G-E^*_{F'}$ and $G-E^*_{\tilde{F'}}$
consist of $k$ and $l$ components induced by $T_i$'s and $\tilde{T}_i$'s,
respectively.
Since $E^*_{\tilde{F'}} = E^*_{F'}$,
$G-E^*_{F'} = G-E^*_{\tilde{F'}}$.
Thus, $k=l$ and we may assume that $V(T_i)=V(\tilde{T}_i)$ for every $i$.
Hence,
an edge $e \in E_{C_1}(F)$ connects $T_i$ and $T_j$ in $G$,
if and only if $e$ connects $\tilde{T}_i$ and $\tilde{T}_j$ in $G$.
Therefore,
since $F$ is a forest,
$H$ also is a spanning forest of $G$ and $E^*_H = E^*_F$.

Since $C'_0 \cap color(\tilde{F'}) = \emptyset$ and $C'_0 = C_0 \cup C_1$,
$\tilde{F'}$ has no colors of $E_{C_1}(F)$.
Thus, $H = \tilde{F'} + E_{C_1}(F)$ is $f$-chromatic
because both $\tilde{F'}$ and $F$ are $f$-chromatic.
Since $C'_0 \cap color(\tilde{F'}) = \emptyset$ and $C'_0 = C_0 \cup C_1$,
$\tilde{F'}$ has no colors in $C_0$.
By the saturated condition (i) of $<F, C_0, C_1>$,
$E_{C_1}(F)$ has no colors in $C_0$.
Hence,
$C_0 \cap color(H) = C_0 \cap color(\tilde{F'} + E_{C_1}(F)) = \emptyset$.
\end{proof}

\begin{claim}
$<F', C'_0, C'_1>$ satisfies
the saturated condition (v).
\label{claim-100704-4}
\end{claim}
\begin{proof}
Suppose that
the triple $<F', C'_0, C'_1>$ does not satisfy
the saturated condition (v),
namely,
there exists some color $c \in C'_1$ such that $|E_c(\tilde{F'})| \ne f(c)$
for some $f$-chromatic spanning forest $\tilde{F'}$ of $G$
such that $E^*_{\tilde{F'}} = E^*_{F'}$
and $C'_0 \cap color(\tilde{F'}) = \emptyset$.
Then,
$|E_c(\tilde{F'})| < f(c)$
because $\tilde{F'}$ is $f$-chromatic.
By the definition of $C'_1$,
$E^*_{F'}$ has some edge $e$ with the color $c$.

First, we show that 
$e \notin E(H)$ and $e \notin E^*_H$.
Since $E^*_{\tilde{F'}} = E^*_{F'}$,
$e \in E^*_{\tilde{F'}}$,
which implies $e \notin E(\tilde{F'})$.
By the definition of $C'_1$ and $C'_0$,
and the saturated condition (ii) of $<F, C_0, C_1>$,
$c \notin C'_0 = C_0 \cup C_1 = color(E^*_F)$,
so $e \notin E_{C_1}(F)$ and $e \notin E^*_F$.
Then, we have shown the following subclaims.

\begin{subclaim}
$c \notin C'_0 = C_0 \cup C_1 = color(E^*_F)$.
\label{subclaim-100705-1}
\end{subclaim}

\begin{subclaim}
$e \in E^*_{\tilde{F'}}$,
$e \notin E(\tilde{F'})$,
$e \notin E_{C_1}(F)$,
and $e \notin E^*_F$.
\label{subclaim-100705-2}
\end{subclaim}

Thus, $e \notin E(\tilde{F'}) \cup E_{C_1}(F) = E(H)$,
and
$e \notin E^*_H$
because $E^*_F = E^*_H$
by Claim \ref{claim-100704-3}.

\begin{subclaim}
$e \notin E(H)$ and $e \notin E^*_H$.
\label{subclaim-100705-3}
\end{subclaim}

Hence, by Subclaim \ref{subclaim-100705-3},
the edge $e$ connects two vertices $x$ and $y$
in the same tree component $T$ of $H$,
and $H+e$ has a cycle $D$.
Since $e \in E^*_{\tilde{F'}}$
by Subclaim \ref{subclaim-100705-2},
$e$ connects two tree components $\tilde{T}_i$ and $\tilde{T}_j$
of $\tilde{F'}$ for some $i$ and $j$.
Note that both $\tilde{T}_i$ and $\tilde{T}_j$ are subgraphs of $T$
because $e \notin E^*_H$ by Subclaim \ref{subclaim-100705-3}.
Thus, there exists a path connecting $x$ and $y$ without $e$ in $H$,
and the cycle $D$ consists of such a path and $e$.
Hence, the cycle $D$ contains some edge
$e' \in E(H) \setminus E(\tilde{F'}) = E_{C_1}(F)$
by the definition of $H$.
Note that $T+e-e'$ is a tree with $V(T)=V(T+e-e')$.
Let $\tilde{F} = H+e-e'$ and $c' = color(e')$.
Then, $c \ne c'$ because $c' \in C_1$ and $c \notin C_1$
by Subclaim \ref{subclaim-100705-1}.
By Claim \ref{claim-100704-3},
$H$ is an $f$-chromatic spanning forest of $G$.
Since $H=\tilde{F'}+E_{C_1}(F)$ and $c \notin C_1$
by Subclaim \ref{subclaim-100705-1},
$|E_c(H)|=|E_c(\tilde{F'})| < f(c)$
by the assumption $|E_c(\tilde{F'})| < f(c)$.
Thus, $\tilde{F}= H+e-e'$ is also an $f$-chromatic spanning forest of $G$.
Since $V(T)=V(T+e-e')$, $E^*_{\tilde{F}} = E^*_{H}$,
that is,
$E^*_{\tilde{F}} = E^*_{F}$ by Claim \ref{claim-100704-3}.
Moreover,
$c \notin C_0$ by Subclaim \ref{subclaim-100705-1}.
Thus,
By Claim \ref{claim-100704-3},
$C_0 \cap color(\tilde{F}) = C_0 \cap color(H+e-e') = \emptyset$.
Hence,
$\tilde{F}$ is an $f$-chromatic spanning forest of $G$
such that $E^*_{\tilde{F}} = E^*_F$
and $C_0 \cap color(\tilde{F}) = \emptyset$.
On the other hand, since $c \ne c' \in C_1$,
$|E_{c'}(\tilde{F})| = |E_{c'}(H)|-1 < |E_{c'}(H)| \le f(c')$,
which contradicts the saturated condition (v) of $<F, C_0, C_1>$.
\end{proof}

By Claim \ref{claim-100704-1}, \ref{claim-100704-2}, and \ref{claim-100704-4},
$<F', C'_0, C'_1>$ is saturated.
As discussed above,
it contradicts the maximality of $C_0$.
Consequently, Theorem \ref{thm-091015-1} is proved.

\section*{Acknowledgments}
I appreciate Hikoe Enomoto
for valuable suggestions for the proof of Theorem \ref{thm-091015-1}.


\end{document}